\documentclass[11pt]{amsart}
\usepackage{amsmath, amsthm, graphicx, amssymb,verbatim,pst-all,color}
\usepackage[all]{xy}
\usepackage[margin=1.0in]{geometry}
\ifx\JPicScale\undefined\def\JPicScale{1.0}\fi
\unitlength \JPicScale mm
{\theoremstyle{plain}
   
   \newtheorem{theorem}{Theorem}[section]
   \newtheorem{proposition}[theorem]{Proposition}     
   \newtheorem{lemma}[theorem]{Lemma}
   
   \newtheorem{corollary}[theorem]{Corollary}

}
{\theoremstyle{definition}
   \newtheorem*{ack}{Acknowledgements}
   
   \newtheorem{example}[theorem]{Example}
   
   \newtheorem{remark}[theorem]{Remark}

}
\newcommand{\quotient}{/\hspace{-1.2mm}/}

\newcommand{\MM}{\overline{\mathcal{M}}}
\newcommand{\M}{\overline{M}}
\newcommand{\SL}{\operatorname{SL}}

\newcommand{\Spec}{\operatorname{Spec}}
\newcommand{\Pic}{\operatorname{Pic}}
\newcommand{\Hom}{\operatorname{Hom}}
\newcommand{\Image}{\operatorname{Im}}
\newcommand{\Cox}{\operatorname{Cox}}
\newcommand{\Bl}{\operatorname{Bl}}

\newcommand{\Ga}{\mathbb{G}_a}
\newcommand{\Gm}{\mathbb{G}_m}
\newcommand{\PP}{\mathbb{P}}

\numberwithin{theorem}{section}

\begin{document}

\title{Projective linear configurations via non-reductive actions}
\author{Brent Doran and Noah Giansiracusa}
\maketitle


\begin{abstract}
We study the iterated blow-up $X$ of projective space along an arbitrary collection of linear subspaces.  By replacing the universal torsor with an $\mathbb{A}^1$-homotopy equivalent model, built from $\mathbb{A}^1$-fiber bundles not just algebraic line bundles, we construct  an ``algebraic uniformization":  $X$ is a quotient of affine space by a solvable group action.  This provides a clean dictionary, using a single coordinate system, between the algebra and geometry of hypersurfaces: effective divisors are characterized via toric and invariant-theoretic techniques.  In particular, the Cox ring is an invariant subring of a $\Pic(X)$-graded polynomial ring and it is an intersection of two explicit finitely generated rings.  When all linear subspaces are points, this recovers  a theorem of Mukai while also giving it a geometric proof and topological intuition.  Consequently, it is algorithmic to describe $\Cox(X)$ up to any degree, and when $\Cox(X)$ is finitely generated it is algorithmic to verify finite generation and compute an explicit presentation.  We consider in detail the special case of $\M_{0,n}$.  Here the algebraic uniformization is defined over $\Spec\mathbb{Z}$.  It admits a natural modular interpretation, and it yields a precise sense in which $\M_{0,n}$ is ``one non-linearizable $\mathbb{G}_a$ away" from being a toric variety.  The Hu-Keel question becomes a special case of Hilbert's 14th problem for $\Ga$.
\end{abstract}


\section{Introduction}


\subsection{The main theorem and its context}

Much of late nineteenth-century geometry was predicated on the presence of a single global coordinate system endowed with a group action.  Indeed, the strong links between representation theory, moduli of projective hypersurfaces,  and invariant theory famously led to the development, from Hilbert's celebrated papers on algebra onwards, of the modern algebro-geometric dictionary.  Moreover, experience has shown that spaces admitting such a description often have beautiful and geometrically meaningful, if intricate, internal structure.  For instance, partial flag varieties can be described via Stiefel coordinates, realizing them as geometric invariant theory (GIT) quotients of affine space.  Many of the striking links between their algebraic geometry (and that of their subvarieties, intersection theory, periods, etc.) and subtle questions in representation theory \cite{CG97,BS00,Res10,LY13,CDGW13} can be seen more concretely in that light.  Another modern class of examples is toric varieties, whose homogeneous coordinate ring and description as a GIT quotient of affine space is powerful and widely used in algebraic geometry and mathematical physics \cite{GKZ94, Cox95, CK99}.  

Within projective geometry, a core question is to understand the geometry of subvarieties of a given variety $X$, e.g., the 27 lines on a cubic surface.  There are two ways in which the algebraic uniformization perspective naturally applies, reflecting the decomposition of moduli of subvarieties into continuous families and discrete ``topological" data.  The first way is to restrict to a given divisor class of hypersurfaces (or, say, rational equivalence class of complete intersections) in $X = \mathbb{A}^n \quotient G$, and consider the space of divisors in the class; two divisors are ``polynomial equivalent" if related by automorphisms of $X$, i.e., by $G$-equivariant automorphisms of $\mathbb{A}^n$.  This perspective, especially the study of simplified versions of the moduli problem, has proved especially fruitful in mirror symmetry with the description of moduli of Calabi-Yau varieties and their periods \cite{GKZ94, CK99, DK07, LY13}.  

The second way the perspective applies not only encodes but also supplements the ``topological" (or, better, ``motivic") constraint organizing subvarieties by their rational equivalence classes.  The difficulty is that identifying which rational equivalence classes admit a representative subvariety, i.e., identifying the effective monoid or cone of a given dimension, is somewhat orthogonal to even our best techniques inspired by topology.  Indeed, with few exceptions (e.g., spherical varieties \cite{FMSS95}) these problems are considered to be very hard.  However, the group action on affine space transfers such effectivity questions to existence of invariants and representation theory questions, which in turn can often be approached by GIT or moment map techniques.  

In this paper we develop a class of examples from projective geometry where, like the toric and flag varieties, there is a ``uniformization" by affine space and we have a perfect algebra-geometry dictionary in the most classical sense.   It will emerge that our construction reflects the surprisingly subtle geometry of linear projections in projective space.  Alternatively, from a linear algebraic perspective, and as will be evident from the explicit description of the group actions, we are focusing on degenerations in families of shear transformations.  

\begin{theorem}\label{intro:quotient}
Let $X$ be an iterated blow-up of $\PP^r$ along a collection of $s$ linear subspaces.  There exists $t\ge 0$ such that $X$ is isomorphic to a (non-reductive) GIT quotient of $\mathbb{A}^{r+s+t+1}$ by an action of $\Ga^{t}\rtimes\Gm^{s+1}$.  This isomorphism is defined over the same ring as the linear subspaces.
\end{theorem}

Varying the choice of character yields many birational models of these spaces.  One motivation for this quotient construction is to lay the foundations for a future exploration of wall-crossings in this more flexible setting.  Theorem \ref{intro:quotient} builds directly on work of the first author \cite{AD07, AD08,AD09,DK07} but the arguments herein are self-contained and mostly quite elementary.

A minimal value of the parameter $t$ measures how far $X$ is from being toric, with $t=0$ being the case that $X$ is in fact a toric variety.  For instance, the blow-up of $s$ points in general linear position in the plane $\PP^2$ corresponds to $t=\max\{0,s-3\}$.  We may thus think of these spaces as being ``almost toric" in a quantifiable sense.

Up to a codimension two modification, the affine space in Theorem \ref{intro:quotient} is the affine envelope for a particular presentation of the geometric $\mathbb{A}^1$-universal covering space of $X$.  One may view this affine space as a master space for the almost toric variety $X$, in the sense of Thaddeus and variation of GIT \cite{Tha96}, but now for a solvable group action, or a torus action ``up to algebraic homotopy", much as affine space actually is a master space for a toric variety.  From this point of view, the existence of infinitely many wall-crossings in general is to be expected.

The class of varieties covered by this theorem includes all del Pezzo surfaces, as well as their higher-dimensional and non-Fano analogues, by taking the linear subspaces to be points.  An interesting setting where blow-ups of positive-dimensional subspaces naturally arise is the De Concini and Procesi theory of wonderful compactifications of subspace arrangements \cite{DP95}.  A special case is the braid arrangement, which corresponds to the moduli space $\M_{0,n}$ of $n$-pointed stable rational curves via Kapranov's construction of this space as an iterated blow-up of $\PP^{n-3}$ \cite{Kap93a,Kap93b}.  This forms central example in this paper.  We use Theorem \ref{intro:quotient} in \cite{DGJ14} to study the Cox ring, effective cone, and monoid of effective divisor classes, over fields of arbitrary characteristic, of $\M_{0,n}$.

Our construction demonstrates that certain phenomena commonly viewed as peripheral, exotic or even pathological are actually at the core of the intricate global geometry of the blow-ups in Theorem \ref{intro:quotient}, and in particular of $\M_{0,n}$.  The first crucial phenomenon is that, unlike in topology, there are many Zariski locally trivial affine space bundles that are not algebraic vector bundles.  For $\M_{0,n}$ there is a distinguished one of rank one on an open subset with codimension two complement; keeping track of this is geometrically important and underlies almost all of the results in this paper and  in \cite{DGJ14}.  A second phenomenon is that unipotent group representations come in positive-dimensional families, unlike the reductive case which is constrained by a discrete classification, and arise naturally in (often infinite-dimensional) automorphism groups of affine varieties; this is reflected in the positive-dimensional moduli of blow-ups of sufficiently many linear subspaces.  A third phenomenon is that almost all automorphisms of affine space are non-linearizable.  Our results strongly suggest that restricting to linear representations is not a geometrically natural thing to do even for the study of low-dimensional smooth projective geometry.  It is interesting to speculate how links between representation theory and algebraic geometry might naturally extend to the non-linearizable case.  


\subsection{A homotopic replacement of the universal torus torsor}\label{IntSec:UT}

From one point of view, we use an extension of the idea of universal torsors due to arithmetic geometers Colliot-Th\'{e}l\`{e}ne and Sansuc and later developed by many others \cite{CT76,CT79,CT87}.  The difference is that they only consider torus torsors, or equivalently over an algebraically closed field, the data of algebraic line bundles on a given variety up to isomorphism.  Yet there are Zariski locally trivial $\mathbb{A}^1$-bundles that are not isomorphic to algebraic line bundles.  Taking into account such $\mathbb{A}^1$-bundles, one can try to replace the universal torsor with a solvable group torsor (or a small modification thereof) whose total space is a much better behaved quasi-affine variety. 

In the context of Theorem \ref{intro:quotient}, this goes as follows.  Label the subspaces defining $X$ as $L_1,\ldots,L_s\subset \PP^r$.  For $t$ large enough, there exists \emph{coordinate} linear subspaces $L_1',\ldots,L_s'\subset\PP^{r+t}$, with $\dim L_i' = \dim L_i$, and a linear projection $\varphi : \PP^{r+t} \dashrightarrow \PP^r$ satisfying $\varphi(L_i') = L_i$.  If $X'$ denotes the corresponding iterated blow-up along the $L_i'$, then $X'$ is a toric variety and the induced rational map $\varphi' : X' \dashrightarrow X$ is a Zariski-locally trivial $\mathbb{A}^t$-bundle over $\Image\varphi'$.  Away from a codimension $\ge 2$ locus, this lifts to a $\Ga^t$-torsor on the universal torsor of $X$; the total space is isomorphic in codimension one to the universal torsor of $X'$ and the fibers are the orbits of a $\Ga^t$-action, which naturally extends to the affine envelope $\mathbb{A}^{r+s+t+1} := \Spec\Cox(X')$.  

In short, by studying $\mathbb{A}^1$-bundles, not just algebraic line bundles, we extend Cox's quotient construction of toric varieties to a large class of varieties traditionally considered very far from toric---and in doing so we allow additive, not just multiplicative, groups and hence leave the realm of Mumford's GIT \cite{GIT} and rely instead on non-reductive GIT \cite{DK07}.


\subsection{Cox rings and effective divisors}

A Cox ring is the ring of functions on the quasi-affine universal torsor of a variety.  These have been difficult to study apart from a few special cases for at least three reasons: (i) this ring need not be finitely generated, (ii) even if it is finitely generated, so the base variety is a Mori dream space \cite{HK00}, any effort to present the ring must overcome the ``false positives", i.e., many separating sub-algebras of functions tantalizingly exist yet give an incorrect affine envelope for the universal torsor upon taking $\Spec$, and (iii) in the rare case a complete presentation of a finitely generated Cox ring can be found, taking $\Spec$ usually gives a high codimension, far from complete intersection, and highly singular affine variety.  Underlying Theorem \ref{intro:quotient} is a new approach to the Cox rings of these blow-ups that for many purposes bypasses all three of these issues.  

We in essence translate the geometry of $X$ into equivariant geometry of the universal torsor over the toric variety $X'$ discussed in \S\ref{IntSec:UT}.  The necessity of a small modification in some cases means this approach is most direct for studying divisors.  For instance, the effective divisors on $X$ can be described by a combination of toric (in effect, topological) and invariant-theoretic techniques:
\begin{theorem}\label{intro:Cox}
With notation as above, $\Cox(X) = \Cox(X')^{\Ga^t}$.
\end{theorem}

Indeed, the Cox ring is graded by the Picard group and the effective divisor classes correspond to the non-zero homogeneous pieces; the crucial point here is then that the inclusion \[\Cox(X) = \Cox(X')^{\Ga} \subseteq \Cox(X')\] is $\Pic(X)=\Pic(X')$ equivariant---so a divisor class on the almost toric variety $X$ is effective if and only if the same class on the toric variety $X'$ admits an invariant section.

An immediate consequence of Theorem \ref{intro:Cox} is that, regardless of finite generation, it is algorithmic to compute generators of $\Cox(X)$ up to a given degree and their relations, and if $\Cox(X)$ is finitely generated then it is algorithmic to verify this \cite{DK08}.

Mukai had previously identified the Cox ring of $X$ as a $\Ga^t$-invariant subring of a polynomial ring in the case that all linear subspaces $L_i$ are points \cite{Muk01}.  His proof is purely algebraic, so our geometric proof is both novel and a substantial generalization of Mukai's.  For an interpretation of this result in terms of wonderful models, see \cite{BM14}.


\subsection{A new view of the moduli space $\M_{0,n}$}

Although $\M_{0,n}$ has recently featured prominently in many areas involving topological (operadic), motivic and period-theoretic (mixed Tate motives and multiple zeta values), and even physics-inspired questions \cite{EHKR10,GM03,Del12,Bro12,Zag12}, its classical global geometry has remained largely a mystery.  This is true despite much study and well-known applications (e.g., its cone of curves determines the ample cone of $\M_{g}$, so describing it would answer an old question of Mumford \cite{Mum77,GKM02}).  Not even the effective divisors for $n \geq 7$ are known, and until now there had been no geometrically meaningful program available to study them, despite surprising evidence showing how far off is our naive understanding of the situation \cite{Ver02,CT13a,Pix11}.

The application of Theorem \ref{intro:quotient} to $\M_{0,n}$ may be succinctly summarized as saying that $\M_{0,n}$ is one non-linearizable $\mathbb{G}_a$-action away from being a toric variety.  This precisely encodes the rough sense practitioners in the area have had that $\M_{0,n}$ is almost toric even though there is no big torus in the automorphism group and no toric techniques apply.  Indeed, since Kapranov blows up the linear subspaces spanned by $n-1$ general points in $\PP^{n-3}$, we can set $t=1$ in the notation of Theorem \ref{intro:quotient} and blow up the dimension $\le n-5$ subspaces spanned by all $n-1$ coordinate points in $\PP^{n-2}$.  The toric variety $X'$ this produces is a Hassett space of weighted pointed curves, closely related to the Losev-Manin space of pointed chains, and the rational map $X' \dashrightarrow \M_{0,n}$ admits a modular interpretation in this context; cf. \S\ref{sec:modular}.   Note that the linear subspaces being blown up here are defined over $\Spec\mathbb{Z}$, so our construction as a quotient of affine space is as well.

Recall that the well-known Mori dream space question of Hu and Keel asks whether $\M_{0,n}$ has finitely generated Cox ring \cite{HK00}.  This has a positive answer for $n \le 6$ by the log Fano condition \cite{BCHM10} (and verified directly in \cite{Cas09}), and very recent work shows it is false for $n \geq 134$ \cite{CT13b}, but the intermediate cases $7 \le n \le 133$ remain wide open.  We now express the Cox ring statement (Theorem \ref{intro:Cox}) in detail for $\M_{0,n}$ and describe some consequences.  To fix notation, let $X'$ be the toric variety from the preceding paragraph and denote its polynomial Cox ring by \[\Cox(X') = k[y_1,\ldots,y_{n-1},(x_I)_{1 \le |I|\le n-4}]\]
where $I \subset \{1,\ldots,n-1\}$.  Here the $y_i$ are homogeneous coordinates on $\PP^{n-2}$ prior to blowing up (i.e., they define the strict transforms of the coordinate hyperplanes) and the $x_I$ are the exceptional divisors.  It will be convenient to set \[z_i := \prod_{I\ni i}x_I.\]

\begin{corollary}\label{intro:M0nCox}
For any $n \ge 5$ and any field $k$ we have:
\begin{enumerate}
\item $\Cox(\M_{0,n}) = k[y_1,\ldots,y_{n-1},(x_I)_{1 \le |I|\le n-4}]^{\Ga}$, where the action is defined by $x_I \mapsto x_I$ and $y_i \mapsto y_i + \lambda z_i$, $\lambda \in \Ga$; this action is linear for $n=5$ and non-linearizable for $n \ge 6$;  the corresponding locally nilpotent derivation is an elementary monomial derivation;
\item $\Cox(\M_{0,n}) = k[y_1,\ldots,y_{n-1},(x_I)_{1 \le |I|\le n-4}] \cap k[(\frac{y_i}{z_i}-\frac{y_j}{z_j})_{1\le i,j\le n-1},(x_I^{\pm 1})_{1 \le |I|\le n-4}]$;
\item if $\Cox(\M_{0,n})$ is finitely generated for a fixed $n$, then it is algorithmic to verify this and compute a presentation; if $\Cox(\M_{0,n})$ is \emph{not} finitely generated, then it is algorithmic to produce an increasing sequence of generators.
\end{enumerate}
\end{corollary}

Part (2) says the Cox ring is determined by a process of clearing denominators.  The algorithmic methods mentioned in (3) are simply standard procedures for computing, and testing termination of, the invariant ring for a $\Ga$-action \cite{GP93,vdE00}.  

Non-linear actions of $\Ga$ on affine space, and their corresponding locally nilpotent derivations, have been an active area of interest for many years \cite{vdE00}.  Elementary monomial derivations have in particular received special attention; several properties that guarantee finite generation of the kernel, and also that guarantee non-finite generation, have been found \cite{Kho01, Kho04, Kur04, Tan07, Kur09}.  Thus Corollary \ref{intro:M0nCox} provides a novel method of approach and places the Hu-Keel question in the context of an extensive body of literature.  This result is also the basis for a study of the monoid of effective divisor classes in $\M_{0,n}$ and the relation between the Cox ring and effective cone of $\M_{0,n}$; see \cite{DGJ14}.  Motivic, topological, and higher-codimension cycle-theoretic investigations are the subject of future work.

The fact that this $\Ga$-action is linear for $n=5$ immediately implies the effective cone of $\M_{0,5}$ is generated by boundary divisors (see \S\ref{sec:boundary}), so the non-linearity of the action for all $n \ge 6$ is precisely what allows for the interesting divisorial geometry observed by Keel, Vermeire, and Castravet-Tevelev \cite{Ver02,GKM02,CT13a} (and observed further in \cite{DGJ14}).  Moreover, by the classical Weitzenb\"ock theorem, any linear $\Ga$-action on affine space (in characteristic zero) has finitely generated ring of invariants, so again Corollary \ref{intro:M0nCox} gives a new perspective of why the Hu-Keel question is indeed an intriguing and non-trivial one.



\subsection{Organization}
We develop our main construction and results entirely in the context of $\M_{0,n}$ first, since essentially all of the phenomena that arise in general are elegantly, and concretely, illustrated here.  We begin in \S2 by recalling some background from Kapranov's blow-up construction and Hassett's moduli of weighted pointed curves.  The heart of the paper is \S3, where we describe the main construction, from a modular perspective, and prove all our results about $\M_{0,n}$.   We conclude in \S4 by showing how an almost trivial extension of the previous section allows for a generalization to arbitrary linear subspace blow-ups.


\begin{ack} 
We thank Dan Abramovich, Aravind Asok, Charles Doran, Angela Gibney, Brendan Hassett, Frances Kirwan, Andrew Kresch, Dave Jensen, Diane Maclagan, Han-Bom Moon, and Rahul Pandharipande for helpful conversations.  We also thank the many people who saw and commented on various aspects of this project in seminar and conference talks since 2012.  The first author was supported in part by the SNF, the second by the SNF and an NSF postdoctoral fellowship and the kind hospitality of the FIM at ETH Z\"{u}rich.  
\end{ack}


\section{Background}


\subsection{Kapranov's projective presentation}\label{section:KapBlow}

Let us recall Kapranov's construction of $\M_{0,n}$ in \cite{Kap93a}.  Given $n$ points $p_1,\ldots,p_n \in \mathbb{P}^{n-2}$ in general linear position, the set of rational normal curves passing through these points is naturally isomorphic to $M_{0,n}$; the closure in the relevant Hilbert scheme or Chow variety is isomorphic to $\M_{0,n}$; and, the universal curve over $\M_{0,n}$ is the restriction of the universal family over this Hilbert scheme/Chow variety \cite[Theorem 0.1]{Kap93a}.  

For each $1\le i \le n$, the set of lines in $\mathbb{P}^{n-2}$ passing through $p_i$ is isomorphic to $\mathbb{P}^{n-3}$.  The map sending each $C\in \M_{0,n}$ via the above realization to the embedded tangent line to $C$ at $p_i$ is a birational morphism $\M_{0,n} \rightarrow \mathbb{P}^{n-3}$, and it is induced by the complete linear series of the tautological class $\psi_i$ on $\M_{0,n}$ \cite[Propositions 2.8, 2.9]{Kap93a}.  Each of the $n-1$ points $p_j$, $j\ne i$, determines a line through $p_i$ and hence a point $q_j \in \mathbb{P}^{n-3}$.  The morphism $\M_{0,n} \rightarrow \mathbb{P}^{n-3}$ decomposes as a composition of blow-ups, centered at these $n-1$ points and all linear subspaces spanned by them \cite[Theorem 4.3.3]{Kap93b}.  After linear change of coordinates one can assume that $n-2$ of these $n-1$ points in $\mathbb{P}^{n-3}$ are coordinate points.  The order Kapranov uses for this iterated blow-up is to first blow up the coordinate points and all coordinate linear subspaces, ordered according to dimension, and then the remaining point and all the remaining linear subspaces.  Hassett shows in \cite[\S6.2]{Has03} that one can treat all points on equal footing and simply blow up all $n-1$ points $q_j$ and the linear subspaces spanned by them, ordered according to dimension.  Moreover, while Kapranov works over $\mathbb{C}$, it follows from \cite[\S6]{Has03} that these isomorphisms between $\M_{0,n}$ and these iterated blow-ups of $\PP^{n-3}$ are defined over $\mathbb{Z}$.

These blow-up constructions yield a natural basis for the Picard group.  Indeed, \[\Pic(\M_{0,n}) = \mathbb{Z}H\bigoplus_{I \subseteq [n-1]}\mathbb{Z}E_I,\] where $H$ is the strict transform of the hyperplane class in $\mathbb{P}^{n-3}$, $E_I$ is the exceptional divisor lying over the linear subspace spanned by $\{q_i\}_{i\in I}$, and the sum runs over subsets $I \subseteq [n-1] := \{1,\ldots,n-1\}$ satisfying $1 \le |I| \le n-4$.

Linear projection from $p_k\in\mathbb{P}^{n-2}$ yields $n-1$ points of $\mathbb{P}^{n-3}$ in general linear position.  This projection sends rational normal curves in $\mathbb{P}^{n-2}$ passing through the original $n$ points to rational normal curves in $\mathbb{P}^{n-3}$ passing through these $n-1$ points; in fact, this induces a morphism $\M_{0,n} \rightarrow \M_{0,n-1}$ which is the forgetful map dropping the $k^{\text{th}}$ marked point and stabilizing \cite[Proposition 2.7]{Kap93a}.  If we fix the class $\psi_i$ on both spaces, $i\ne k$, then this morphism is identified with a map of iterated blow-ups $\Bl\PP^{n-3} \rightarrow \Bl\PP^{n-4}$ which, upon blowing down, corresponds to linear projection $\mathbb{P}^{n-3} \dashrightarrow \mathbb{P}^{n-4}$ from $q_k$.

The last ingredient to recall is a description of the boundary divisors.  Choose $\psi_n$ to identify $\M_{0,n}$ with the blow-up of $\mathbb{P}^{n-3}$ at $q_1,\ldots,q_{n-1}$ and the linear spaces they span.  The boundary divisor parameterizing curves with a node separating $\{p_i\}_{i\in I} \cup \{p_n\}$ from the remaining marked points is the exceptional divisor $E_I$ if $|I| \le n-4$, and it is the strict transform of the linear subspace $\langle q_i \rangle_{i\in I}$ if $|I| = n-3$.


\subsection{Weighted pointed curves}
Hassett introduced in \cite{Has03} a family of compactifications of the moduli space of marked curves, $\mathcal{M}_{g,n}$, that are smooth, proper Deligne-Mumford stacks over $\Spec \mathbb{Z}$ admitting projective coarse moduli schemes.  These stacks are denoted by $\MM_{g,\mathcal{A}}$, where $\mathcal{A} = (a_1,\ldots,a_n)\in \mathbb{Q}^n$ satisfies $0 < a_i \le 1$ and $2g-2+|\mathcal{A}| > 0$.  The vector $\mathcal{A}$ is viewed as an assignment of a rational weight to each marked point, and this weight data determines which collisions of points are allowed.  More precisely, a Hassett $\mathcal{A}$-stable curve $(C,p_1,\ldots,p_n)$ is a connected, nodal curve such that i) the marked points $p_i$ avoid the nodes, ii) the weights satisfy $\sum_{i \in I}a_i \le 1$ if $\{p_i\}_{i\in I}$ coincide, and iii) the ``log canonical'' divisor $K_C+\sum_{i=1}^n a_ip_i$ is ample.  The usual Deligne-Mumford compactification is obtained by taking $\mathcal{A} = (1,\ldots,1)$.  If $\mathcal{A}'  = (a'_1,\ldots,a'_n)$ satisfies $a'_i \le a_i$ for all $i$ then there is a \emph{reduction} morphism $\MM_{g,\mathcal{A}} \rightarrow \MM_{g,\mathcal{A}'}$ which in modular terms is defined by contracting all irreducible components of a curve for which the log canonical divisor loses ampleness upon reducing the weights from $\mathcal{A}$ to $\mathcal{A}'$.  There are also forgetful morphisms given by dropping marked points and then stabilizing.  See \cite{Has03}.

In genus zero, $g=0$, the ampleness condition in the definition of stability is equivalent to the following: for each irreducible component $C' \subseteq C$, the total weight of marked points on $C'$ plus the number of nodes on $C'$ must be strictly greater than two.  In particular, if $|\mathcal{A}| \le m$ for some $m \in \mathbb{Z}$, then $\mathcal{A}$-stable curves can have at most $m-1$ \emph{tails}, i.e., components with a single node.  A well-studied example of a Hassett space parameterizing marked \emph{chains} of projective lines---that is, curves having at most two tails---is the so-called Losev-Manin space \cite{LM00}.  This is the toric variety associated to the fan of Weyl chambers for the type A root system \cite{BB11}.  Hassett shows in \cite[\S6.4]{Has03} that this space is of the form $\M_{0,\mathcal{A}}$ for an appropriate weight vector $\mathcal{A}$ and that it is the final toric variety appearing in Kapranov's iterated blow-up construction of $\M_{0,n}$.


\section{Main construction}

In this section we explain how Hassett's moduli spaces of weighted pointed curves provide a natural modular interpretation for the ``toric model'' of $\M_{0,n}$ that we are interested in, and then we describe the relation between this model and $\M_{0,n}$ in terms of i) the projective geometry of a rational map, ii) the moduli functors, and iii) a unipotent action relating universal torsors.


\subsection{Toric model as a Hassett space}

Let $X_n$ denote the iterated blow-up of $\mathbb{P}^{n-2}$ along all coordinate points $q_1,\ldots,q_{n-1}\in\mathbb{P}^{n-2}$ and all coordinate linear subspaces, ordered by increasing dimension, up to codimension three.  Recall that to construct $\M_{0,n+1}$ we would also have blown-up the codimension two linear subspaces, yielding Losev-Manin, then all linear spaces spanned by these and an additional non-coordinate point, without loss of generality $q_n = (1: \cdots : 1)$.  The variety $X_n$ can be viewed as a Hassett compactification of $M_{0,n+1}$, namely, $X_n \cong \M_{0,\mathcal{A}}$ for $\mathcal{A} = (\frac{1}{n-1},\ldots,\frac{1}{n-1},\frac{n-3}{n-1},1)$ \cite[\S6.1]{Has03}.  Here the last marked point, $p_{n+1}$, corresponds to the choice of $\psi$-class, while the penultimate marked point, $p_n$, corresponds to the non-blown-up point $q_n\in\mathbb{P}^{n-2}$.  Note that $X_n$ parameterizes chains, since $|\mathcal{A}| = 2 + \frac{n-3}{n-1} \le 3$, and there is a reduction morphism $\M_{0,n+1} \rightarrow X_n$.


\subsection{Forgetting a point}

The forgetful morphism $\M_{0,n+1} \rightarrow \M_{0,n}$ descends to a rational map from each Hassett compactification of $M_{0,n+1}$.  By choosing the weight vector $\mathcal{A}$ above and forgetting the penultimate point $p_n$, we obtain the diagram 
\begin{equation}\label{diag:ratmap}\xymatrix{ & \M_{0,n+1} \ar[dl] \ar[dr] & \\ X_n = \M_{0,\mathcal{A}} \ar@{-->}[rr]^{\varphi_n} & & \M_{0,n} }\end{equation}
which becomes, upon choosing the last $\psi$-class in all three spaces,
\[\xymatrix@C-50pt{ & \Bl_{\mathrm{codim}~2}\cdots\Bl_{n \mathrm{pts}} \mathbb{P}^{n-2} \ar[dl] \ar[dr] & \\\Bl_{\mathrm{codim}~3}\cdots\Bl_{n-1 \mathrm{pts}} \mathbb{P}^{n-2} \ar@{-->}[rr]^{\varphi_n} & & \Bl_{\mathrm{codim}~2}\cdots\Bl_{n-1 \mathrm{pts}} \mathbb{P}^{n-3}.}\]

As above, we label the blown-up points $q_1,\ldots,q_n \in \mathbb{P}^{n-2}$ for $\M_{0,n+1}$ and $q_n$ is the point not blown-up for $X_n$.  As mentioned earlier, upon blowing down all exceptional divisors this rational map is simply the linear projection from $q_n \in\mathbb{P}^{n-2}$.  In fact, a similar interpretation holds prior to blowing down.

\begin{lemma}\label{lem:linproj}
The rational map $\varphi_n$ is induced by linear projection from $q_n$ and in $E_I$ from the strict transform of the linear subspace $\langle q_n, \{q_i\}_{i\in I} \rangle$.  The base locus is the union of the strict transforms of these linear subspaces for $|I| = n-4$.
\end{lemma}

\begin{proof}
This is elementary projective geometry.
\end{proof}

\begin{example}
The rational map $\varphi_5: X_5 = \M_{0,(\frac{1}{4},\frac{1}{4},\frac{1}{4},\frac{1}{4},\frac{1}{2},1)} \dashrightarrow \M_{0,5}$ takes the form $\Bl_{4 \mathrm{pts}}\mathbb{P}^3 \dashrightarrow \Bl_{4 \mathrm{pts}}\mathbb{P}^2$, and the linear projection is from a point $q_5$ in $\mathbb{P}^3$ as well as the unique point of each exceptional $E_i \cong \mathbb{P}^2$ given by the strict transform of the line $\overline{q_iq_5}$.  Note that this entire line is in the base locus since under projection from $q_5$ it gets sent to the projection point of $E_i$.

The map $\varphi_6 : \Bl_{10 \mathrm{lines}}\Bl_{5 \mathrm{pts}}\mathbb{P}^4 \dashrightarrow \Bl_{10 \mathrm{lines}}\Bl_{5 \mathrm{pts}}\mathbb{P}^3$ is as follows: in the base $\mathbb{P}^4$ we project from $q_6$; over each $q_i$, $1\le i \le 5$, the exceptional divisor is $E_i \cong \Bl_{4 \mathrm{pts}}\mathbb{P}^3$ and the line $\overline{q_iq_6}$ meets this at a single point which we project from, whereas the four planes $\langle q_i,q_j,q_6 \rangle$, $1 \le j \le 5$, $j \ne i$, meet $E_i$ in the four lines between this projection point and the four blown-up points in $E_i$; finally, for each of the ten blown-up lines we have the exceptional divisor $E_{ij}$ which is a $\mathbb{P}^2$-bundle and we obtain a $\mathbb{P}^1$-bundle by projecting linearly from the line inside it given by the strict transform of the plane $\langle q_i,q_j,q_6\rangle$, which in particular yields linear projection from a point in $E_{ij} \cap E_i \cong \mathbb{P}^2$.  Note that $\varphi_6$ restricted to $E_i \cong X_5$ is identical to $\varphi_5$.
\end{example}


\subsection{Modular interpretation}\label{sec:modular}

The map $\varphi_n$ forgets the marked point $p_n$ then stabilizes.  The locus where this is not defined admits a simple modular description.

\begin{proposition}\label{prop:baseloc}
The base locus of $\varphi_n : X_n = \M_{0,\mathcal{A}} \dashrightarrow \M_{0,n}$ is the locus of curves such that at least three points among the first $n-1$ coincide.
\end{proposition}

\begin{proof}
This follows from Lemma \ref{lem:linproj}.  Linear subspaces of the form  $\langle q_n, \{q_i\}_{i\in I} \rangle$ for $|I| = n-4$ are precisely the image under the reduction morphism $\M_{0,n+1} \rightarrow X_n$ of the boundary divisors $E_I$ parameterizing curves with a node separating points indexed by $I \cup \{p_n,p_{n+1}\}$ from the remaining three marked points.  The component containing these three marked points is contracted by this reduction blow down.
\end{proof}
 
Fix a point $(C,p_1,\ldots,p_n) \in \M_{0,n}$ and consider diagram (\ref{diag:ratmap}).  The fiber in $\M_{0,n+1}$ is isomorphic to $C$ and its image in $X_n$ is a partial contraction of $C$.  The intersection of this curve with the complement of the base locus is the fiber under $\varphi_n$.  This fiber also admits a modular interpretation.

\begin{proposition}\label{prop:fiber}
Each non-empty fiber of $\varphi_n$ is isomorphic to $\mathbb{A}^1$.   If $(C,p_1,\ldots,p_n)\in \M_{0,n}$ is an interior point, namely $C \cong \mathbb{P}^1$, then the fiber is naturally identified with $C\setminus\{p_{n+1}\}$.  If $C$ is nodal, let $C'\subseteq C$ denote the irreducible component carrying the marked point $p_{n+1}$, and consider the connected components of $\overline{C\setminus C'}$, each of which must contain at least two marked points.
\begin{itemize}
\item If no connected component carries more than two marked points then the fiber is $C'\setminus \{p_{n+1}\}$.
\item If there are at least two connected components each carrying more than two marked points, then the fiber is empty.
\item If exactly one connected component carries more than two marked points, then either the fiber is empty---which can only occur if $C$ has at least three nodes---or the reduction $\M_{0,n+1} \rightarrow X_n$ contracts $C$ down to a chain, say $C' \cup C_1 \cup \cdots \cup C_m$, and the $\varphi_n$-fiber is $C_m$ with its unique node removed.
\end{itemize}
\end{proposition}

\begin{proof}
The fiber over $(C,p_1,\ldots,p_n)\in \M_{0,n}$ in $\M_{0,n+1}$ is obtained by letting $p_n$ trace out the curve $C$.  Applying the reduction morphism yields a 1-parameter family of $\mathcal{A}$-stable curves, again with $p_n$ the moving point.  If $C \cong \mathbb{P}^1$, then it follows immediately from Proposition \ref{prop:baseloc} that i) for $p_n$ distinct from the other marked points the corresponding point of $X_n$ is not in the base locus, ii) when $p_n$ collides with a marked point $p_i$, $i\ne n+1$, again the corresponding point of $X_n$ is not in the base locus, and iii) the point of $X_n$ corresponding to $p_n$ colliding with $p_{n+1}$ is in the base locus, since in $\M_{0,n+1}$ this collision results in a 2-component curve, with marked points $p_n,p_{n+1}$ on one component and the remaining points on the other component, but these remaining points have total weight $(n+1)\frac{1}{n-1}=1$ so the reduction morphism contracts the component they lie on.  

If $C$ is nodal but each connected component of $\overline{C\setminus C'}$ has only two marked points, then these connected components are necessarily irreducible tails which are contracted by the reduction morphism, since the total weight they carry when $p_n$ traces them out is $\frac{n-3}{n-1}+(2)\frac{1}{n-1} = 1$, and then the same argument as in the smooth case shows that the fiber is $C'\setminus \{p_{n+1}\}$.

If there are multiple connected components of $\overline{C\setminus C'}$ carrying more than two marked points each, then as $p_n$ traces out the curve there will always be a collision of at least three points among $p_1,\ldots,p_{n-1}$ upon applying the reduction morphism, so the entire curve lies in the base locus and hence the fiber is empty.

Finally, consider the case that among the connected components of $\overline{C\setminus C'}$ there is a unique one, call it $D$, with at least three marked points.  The image under reduction of each irreducible component not in $D$ lies in the base locus, since $D$ is contracted to a point as $p_n$ traces out such a component.  As for $D$ itself, we may view the node attaching it to $C'$ as a marked point of weight one lying on an irreducible component $D' \subseteq D$.  Then by considering the connected components of $\overline{D\setminus D'}$ we can apply the same analysis as above, this time with the distinguished node on $D'$ playing the role of the $\psi$-point $p_{n+1}$ on $C'$, since both have weight one from a Hassett-stability perspective.  The final statement of the proposition follows by applying this analysis inductively.
\end{proof}

An immediate consequence is the following:

\begin{corollary}\label{cor:image}
The rational map $\varphi_n$ is surjective for $n \le 6$, and the complement of its image is codimension 2 for $n \ge 7$.
\end{corollary}

\begin{remark}
The lack of surjectivity is essentially due to the fact that for $n \ge 7$ there exist disjoint linear subspaces blown-up to produce $X_n$ such that the corresponding linear subspaces blown-up to produce $\M_{0,n}$ are not disjoint.  For instance, in $\M_{0,7}$ there is a $\mathbb{P}^1\times\mathbb{P}^1$ lying above the point $\langle q_1,q_2,q_3\rangle \cap \langle q_4,q_5,q_6\rangle \in \mathbb{P}^4$.  The generic point of this surface corresponds to a curve with three components: two tails carrying $p_1,p_2,p_3$ and $p_4,p_5,p_6$, respectively, and an inner component carrying the $\psi$-class point $p_7$.  By Proposition \ref{prop:fiber}, this $\mathbb{P}^1\times\mathbb{P}^1$ lies outside the image of $\varphi_7$.
\end{remark}


\subsection{Lifting to the universal torsor and the main theorems}\label{sec:lift}

The universal torsor of $X_n$, as with all toric varieties, is naturally an open subset of affine space---namely Spec of the polynomial Cox ring.  If we denote the Cox ring of $\mathbb{P}^{n-2}$ by $k[y_1,\ldots,y_{n-1}]$, then we have \[\Cox(X_n) = k[y_1,\ldots,y_{n-1},(x_I)_{|I|\le n-4}]\] where $I \subseteq [n-1]$ and each locus $\{x_I=0\}$ corresponds to the exceptional divisor over the subspace spanned by $\{q_i\}_{i\in I}$.  

\begin{theorem}\label{thm:Cox} For any $n \ge 5$ we have $\Cox(\M_{0,n}) = \Cox(X_n)^{\mathbb{G}_a}$ where \[x_I \mapsto x_I, ~y_j \mapsto y_j + \lambda\prod_{I\ni j}x_I, ~\lambda\in \mathbb{G}_a.\]  This action is non-linearizable for $n \ge 6$.
\end{theorem}

\begin{proof}
We will show that, away from a codimension two locus, the universal torsor of $X_n$ is a $\Ga$-torsor over the universal torsor of $\M_{0,n}$.  Since Cox rings are insensitive to codimension two modifications, this suffices.  We build this $\Ga$-torsor in essence by lifting the rational map $\varphi_n : X_n \dashrightarrow \M_{0,n}$ to the universal torsors by choosing a local section.  Any dense open domain of this local section suffices, so we will use the complement of the exceptional divisors, which reduces everything simply to linear projection from a point.  We now begin this argument in earnest.

The polynomial ring $\Cox(X_n)$ admits a natural action by the Neron-Severi torus \[T_{NS} = \Hom(\Pic(X_n),\mathbb{G}_m) = \Hom(\Pic(\M_{0,n}),\mathbb{G}_m)\] and it is a straightforward computation in toric geometry that the character corresponding to the hyperplane class \[H \in \Pic(X_n) = \chi(T_{NS})\] has an eigenbasis given by the monomials \[m_j := y_j\prod_{I\not\ni j} x_I,~j=1,\ldots,n-1.\]  Indeed, this follows from the fact that the divisor $\{y_j = 0\}$ corresponds to the strict transform of the coordinate hyperplane in $\mathbb{P}^{n-2}$ not containing the coordinate point $q_j$.  Thus the projective space $\mathbb{P}^{n-2}$ blown-up to construct $X_n$ naturally has homogeneous coordinates $(m_1 : \cdots : m_{n-1})$.

The rational map $\mathbb{P}^{n-2} \dashrightarrow \mathbb{P}^{n-3}$ induced by $\varphi_n$ is linear projection from the point $q_n \in \mathbb{P}^{n-2}$.  The fibers of this are affine lines, which when expressed in homogeneous coordinates are of the form $(m_1+\lambda,\ldots,m_{n-1}+\lambda)$ for $\lambda\in \Ga$.  In other words, the uniform translation action of $\mathbb{G}_a$ on $\mathbb{A}^{n-1}$ admits a quotient map $\mathbb{A}^{n-1} \rightarrow \mathbb{A}^{n-2}$ that when restricted to the universal torsor (i.e., quasi-affine cone) of $\mathbb{P}^{n-2}$ descends to our rational map $\mathbb{P}^{n-2} \dashrightarrow \mathbb{P}^{n-3}$.  Over the locus in $\mathbb{P}^{n-2}$ consisting of the complement of the coordinate linear subspaces being blown-up, all the $x_I$ coordinates in $\Cox(X_n)$ are nonzero, so we can use homogeneous coordinates \[(m_1\prod_I x_I^{-1} : \cdots : m_{n-1} \prod_I x_I^{-1}).\] to obtain sections, over this generic locus, of the universal torsor of $X_n$.  Now \[m_j \prod_I x_I^{-1} = y_j \prod_{I \ni j} x_I^{-1},\] so the fibers of $\varphi_n : X_n \dashrightarrow \M_{0,n}$ generically lift to orbits under the $\mathbb{G}_a$-action on $\mathbb{A}^{N} := \Spec\Cox(X_n)$ given by $x_I \mapsto x_I$, for all $I$, and \[y_j\prod_{I\ni j} x_I^{-1} \mapsto y_j\prod_{I\ni j} x_I^{-1} + \lambda,\] or equivalently, \[y_j \mapsto y_j + \lambda\prod_{I\ni j}x_I.\]  

We claim this $\mathbb{G}_a$-action provides the structure of a torsor away from its fixed-point locus, which manifestly has codimension at least two.  To avoid detailed discussions of non-reductive stability and semi-stability (see \cite{DK07}), we give a quick self-contained argument that uses just one collection of easy, explicit invariants.  The torsor, i.e., principal bundle, structure follows if the $\mathbb{G}_a$-action off the fixed-point locus yields a geometric quotient, by \cite[Proposition 0.8]{GIT} applied to the $\SL_2$-sweep (sometimes called ``homogeneous fiber space,'' cf. \cite[\S4.8]{PV94} and \cite[\S5]{DK07}) of this locus: the $\SL_2$-quotient map is affine because reductive GIT quotients are always affine morphisms, which yields an $\SL_2$-torsor since the action is free and proper, and then by \'{e}tale descent the additive action is a $\mathbb{G}_a$-torsor.  To show the $\mathbb{G}_a$-action off the fixed-point locus is a geometric quotient, it suffices to show the $\mathbb{G}_a$-orbits are separated by invariants.  Proposition \ref{prop:boundary} below explicitly describes the invariants associated to the boundary divisors, and it is easily checked that these already suffice to separate orbits off the fixed point locus.  

We now claim that (away from a codimension two locus) this $\Ga$-torsor i) inherits the structure of a $\Ga \rtimes T_{NS}$-torsor, ii) has total space admitting no non-trivial line bundles, and iii) has $\M_{0,n}$ as its base.  These claims suffice to finish the proof that $\Cox(\M_{0,n}) = \Cox(X_n)^{\Ga}$, since the Cox ring is the coordinate ring of the universal torsor (and is insensitive to codimension two modifications).  Item i) follows from the fact that the $\Ga$-action on $\Cox(X_n)$ is normalized by the $T_{NS}$-action, item ii) from the fact that $\Spec\Cox(X_n)$ is affine space and the Picard group is an $\mathbb{A}^1$-homotopy invariant, and iii) from the fact that first quotienting by the $\Ga$-action and then by the $T_{NS}$-action is the same as first quotienting by the $T_{NS}$-action, thus yielding $X_n$, and then applying the rational map $\varphi_n$, thus yielding $\M_{0,n}$.

This $\Ga$-action on affine space is non-linearizable for $n \ge 6$ because the fixed point locus for a linear $\mathbb{G}_a$-action is irreducible, corresponding to certain coordinates vanishing.  The fixed point set for our $\mathbb{G}_a$-action clearly has multiple irreducible components, when $n \ge 6$. But the number of irreducible components of the fixed point set is preserved under equivariant isomorphism of affine space.
\end{proof}

\begin{corollary}
The fibers of $\varphi_n : X_n \dashrightarrow \M_{0,n}$ form a Zariski-locally trivial $\mathbb{A}^1$-fibration on $\Image\varphi_n$ that is not an algebraic line bundle.
\end{corollary}

\begin{proof}
That this is a Zariski locally trivial $\mathbb{A}^1$-bundle follows at heart from the fact that $\Ga$ is a special group in the sense of Serre, so that any $\Ga$-torsor is Zariski locally trivial.  Similarly $\Ga \rtimes \mathbb{G}_m^{\rho}$ is a special group.  Thus this solvable group torsor is Zariski locally trivial.  Pick a trivialization.  Upon quotienting by the torus factor of the solvable group, on each trivialization the fiber is an affine line.  The transition functions are necessarily valued in automorphisms of the affine line, giving a Zariski locally trivial $\mathbb{A}^1$-bundle on $\Image\varphi_n$.  If this were an algebraic line bundle, then it would pull-back to an algebraic line bundle on a codimension 2 complement of the universal torsor of $\M_{0,n}$.  But all line bundles on the universal torsor are trivial; and since $\M_{0,n}$ is non-singular, so is this $T_{NS}$ torsor, hence any line bundle on a codimension 2 complement would canonically extend to a line bundle over the full torsor, and thus there is only the trivial line bundle, up to isomorphism, on the codimension 2 complement as well.  Therefore if the $\mathbb{A}^1$-fibration were an algebraic line bundle it would have to be trivial.  But its pull-back to the $T_{NS}$-torsor is, by the Theorem, a non-trivial $\Ga$-torsor.  Thus the $\mathbb{A}^1$-bundle is not an algebraic line bundle.  
\end{proof}

The remaining statements in Corollary \ref{intro:M0nCox} are immediate.  Indeed, that this $\Ga$-action corresponds to an elementary monomial derivation is by definition, and the claim about algorithmics, as mentioned in the introduction, hold for any $\Ga$-invariant subring---so all that remains is item (2), the denominator-clearing interpretation of the Cox ring.  This follows from the fact that taking invariants commutes with localization (when inverting invariant functions, such as the $x_I$) and the observation that when all $x_I$ are inverted the induced action is uniform translation $\frac{y_i}{z_i} \mapsto \frac{y_i}{z_i} + \lambda$, so the invariants are indeed given by these differences.

In fact, the $\M_{0,n}$ case of Theorem \ref{intro:quotient} also follows from the above Cox ring result.  Indeed, one could approach this by noting that any projective variety $X$ is the GIT quotient of $\Spec\Cox(X)$ by its Neron-Severi torus for an ample character, and viewing a succession of GIT quotients (additive then torus) as a single GIT quotient for the full group, a perspective compatible with \cite{DK07} and further developed in \cite{DHK14}; when the (semi)invariants for the full group are finitely generated one simply takes Proj.  To stay self-contained, let us recall that standard GIT identifies a homogeneous coordinate ring of the projective quotient variety (which depends on a choice of very ample line bundle) with the subring of (semi)invariant sections of (powers of) a line bundle on the original variety.  With this as the definition of a non-reductive GIT quotient variety, i.e., that for a suitable very ample line bundle $L$ on the projective quotient, there is a line bundle on the pre-quotient and a choice of character such that the graded ring of semi-invariants for multiples of that character is isomorphic to the graded ring generated by sections of $L$, we see:

\begin{corollary}\label{cor:affinequot}
The variety $\M_{0,n}$ is a non-reductive GIT quotient of affine space by an action of the solvable group $\Ga\rtimes \Gm^{\rho}$, where $\rho$ is the Picard number.  
\end{corollary}

\begin{proof}
Pick a very ample line bundle on $\M_{0,n}$.  Upon taking a suitable tensor power, call it $L$, this may be taken without loss of generality to be projectively normal such that multiplication maps $H^0(\M_{0,n}, L^p) \otimes H^0(\M_{0,n}, L^q) \longrightarrow H^0(\M_{0,n}, L^{pq})$ are surjective.  
Consider the canonically associated character $\chi_L$ for $T_{NS}$, which determines, by the theorem, a multi-graded piece of not only $\Cox(\M_{0,n})$ but also of the polynomial ring $\Cox(X^{\prime})$.    

The space of sections $H^0(\M_{0,n},L)$ is identified with the $\chi_L$-multi-graded subspace of $\Cox(\M_{0,n})$ by definition, and then with the $\Ga$-invariant elements in $\Cox(X^{\prime})$ by the Theorem.  The ring $R$ generated by taking symmetric powers of the space of  sections of $H^0(\M_{0,n}, L)$ therefore admits a canonical evaluation homomorphism to a subring of the polynomial ring $\Cox(X^{\prime})$ whose elements are solvable group semi-invariants for some $m \chi_L$, where $m$ is a positive integer, by the Theorem.  We simply want to show that conversely any semi-invariant in the polynomial ring $\Cox(X^{\prime})$ for a character $m \chi_L$ lies in the image of $R$ under the evaluation homomorphism.  But this follows immediately from our assumptions on $L$:  any such semi-invariant is automatically an element of $\Cox(\M_{0,n})$ by the theorem, and by surjectivity of multiplication maps it must be polynomial in the elements of the $\chi_L$-multi-graded piece of $\Cox(\M_{0,n})$ and hence is identified as the image of an element of $R$.  
\end{proof}

\begin{remark}
Observe that, from an invariant theory perspective, these solvable group semi-invariants can be thought of as arising from a two stage process, in essence corresponding to a two-stage GIT quotient but with possible non-finite generation issues: first restrict to the subring of $\Ga$-invariants of $\Cox(X^{\prime})$, then restrict further to the subring of $\Ga$-invariants that are admissible semi-invariants for the torus action with chosen character.  Furthermore, using the definitions from \cite{DK07} and \cite{DHK14}, one shows that the non-reductive GIT quotient construction works as in the reductive case for any choice of line bundle and ample character on the pre-quotient, here affine space, so there is no need to specialize to appropriate tensor powers of $L$ as in the proof of the theorem.  However, we will not make use of this slightly stronger form of the result in this paper, as we are concerned with the quotient variety not the descended line bundle.  
\end{remark}

\begin{remark}
Another way of expressing this is that for a certain open subset of this affine space $\Spec\Cox(X_n)$, corresponding to the semistable points of the group action (and choice of ample character) under a suitable definition \cite{DK07,DHK14}, $\M_{0,n}$ is the categorical quotient in the category of varieties.  Note this has the feature that the quotient map need not be surjective, but rather may have as image a constructible set instead of a variety.  Of course, the categorical quotient still exists as a variety, containing the constructible set as a dense subset, by definition.
\end{remark}

\begin{remark}
Since the $\Ga$-action discussed in this section is defined over $\Spec\mathbb{Z}$, as is the isomorphism $\M_{0,n} \cong \Bl\PP^{n-3}$, it is easy to see that all proofs and results in this section are valid over $\Spec\mathbb{Z}$.
\end{remark}


\subsection{Boundary divisors and invariants}\label{sec:boundary}

To illustrate the identification, provided by Theorem \ref{thm:Cox}, of $\Cox(\M_{0,n})$ with a subring of a polynomial ring, in this section we identify the polynomials corresponding to sections of the boundary divisors.  In particular, this gives a new proof (more importantly, conceptual reason) why the boundary divisors generate the effective cone for $n = 5$.

\begin{proposition}\label{prop:boundary}
The elements of $\Cox(\M_{0,n}) =\Cox(X_n)^{\mathbb{G}_a}$ corresponding to sections of the boundary divisors are exactly the variables $x_I$ and the binomials \[y_i \prod_{I\ni j, I\not\ni i}x_I - y_j \prod_{I\ni i, I\not\ni j}x_I\] for $i, j \in [n-1]$ with $i\ne j$.
\end{proposition}

\begin{proof}
Since the rational map $\varphi_n$ sends the exceptional divisors of $X_n$ to the corresponding exceptional divisors of $\M_{0,n}$, it is immediate that $x_I$ is the unique (up to scalar) section of $E_I$.  The other boundary divisors are given by the strict transform of the hyperplanes spanned by $\langle q_i\rangle_{i\in I}$ for $I\subseteq [n-1]$ satisfying $|I|=n-3$, as explained in \S\ref{section:KapBlow}.  Each such hyperplane is the hyperplane passing through all but two of the distinguished points in $\mathbb{P}^{n-3}$, say $q_i$ and $q_j$.  The class of the strict transform is then \[H - \sum_{I \subseteq [n-1]\setminus \{i,j\}}E_I.\]  This class yields a character of $T_{NS}$ with two-dimensional eigenspace in $\Cox(X_n)$, spanned by the two monomials appearing in the proposition statement.  Clearly there is a one-dimensional subspace of $\mathbb{G}_a$-invariant polynomials and it is spanned by the binomial that is the difference of these monomials.
\end{proof}

It is well-known and easy to verify that the effective cone of $\M_{0,5}$ is spanned by boundary divisor classes (see, e.g., \cite[\S5.3]{CFM12}).  Our description of the Cox ring yields another perspective on this fact.  Indeed, we can always describe the $\mathbb{G}_a$-invariant ring in Theorem \ref{thm:Cox} as an $\SL_2$-invariant ring by taking the sweep $\SL_2\times_{\mathbb{G}_a}\Spec\Cox(X_n)$ and considering its natural $\SL_2$-action; the key point is that for $n=5$ the $\mathbb{G}_a$-action is linear so this sweep is an open subset (with codimension two complement) of $\mathbb{A}^{10}$, since $\SL_2/\mathbb{G}_a \cong \mathbb{A}^2\setminus\{0\}$ and $\Spec\Cox(X_5) \cong \mathbb{A}^8$, and the $\SL_2$ action is given by a direct sum of the defining representation.  The invariants are therefore the $2\times 2$ minors of the $2\times 5$ coordinate matrix, and when restricting back to the original $\mathbb{A}^8$ these are $x_1,\ldots,x_4$ and the six binomials $y_ix_j - y_jx_i$ appearing in Proposition \ref{prop:boundary}---so they are exactly the boundary divisor invariants.  A discussion of generators beyond these boundary invariants, for $n > 5$, appears in \cite{DGJ14}.


\section{Arbitrary configurations of linear subspaces}

Here we discuss the natural generalization of the preceding section to the case of an arbitrary collection of linear subspaces, say $L_1,\ldots,L_s\subset \PP^r$, and thereby prove Theorems \ref{intro:quotient} and \ref{intro:Cox}.  Choose a basis $p^i_1,\ldots,p^i_{\dim L_i+1}$ of each $L_i$, and let \[n := \sum_{i=1}^s (\dim L_i+1).\]

If $n \le r+1$, then a linear change of coordinates transforms the $L_i$ to coordinate linear subspaces and we see that the iterated blow-up $X = \Bl\PP^r$ along the $L_i$ is a toric variety and there is nothing to prove.  So assume $n > r+1$, and set $t := n - r - 1$.  Let $A\in M_{(r+1)\times n}$ be a matrix with columns given by homogeneous coordinates of the $p_j^i$.  Then $A$ induces a linear projection $\varphi_A : \PP^{r+t} \dashrightarrow \PP^r$ with indeterminacy locus $K := \PP\ker A$ sending the $n$ coordinate points to the $n$ points $p_j^i$.  For each $1 \le i \le s$, let $L_i' \subseteq \PP^{r+t}$ be the linear subspace spanned by the coordinate points mapping to $p^i_1,\ldots,p^i_{\dim L_i+1}$.  Clearly $\dim L'_i = \dim L_i$.  We claim that $L'_i \cap K = \varnothing$, and hence $\varphi_A(L'_i) = L_i$.  Indeed, this follows immediately from the fact that the points $p^i_1,\ldots,p^i_{\dim L_i+1}$ were chosen to be linearly independent, since $L'_i$ meets $K$ if and only if the matrix formed by the columns of $A$ corresponding to these points has non-trivial kernel.

Consider the $\Ga^{n}$-action on $\mathbb{A}^{n}$ defined by $y_i \mapsto y_i + \lambda_i$, where $\lambda_i$ is a parameter for the $i^{\text{th}}$ factor, and let $G := \ker A \subseteq \Ga^n$.  Then the above linear projection $\varphi_A$ is induced by the quotient $\mathbb{A}^n \rightarrow \mathbb{A}^n/G$ and it extends to a rational map \[X' = \Bl\PP^{r+t} \dashrightarrow \Bl\PP^r = X\] from the iterated blow-up along the $L'_i$.  Note that $X'$ is a toric variety satisfying $\Pic(X') = \Pic(X)$.  The proofs of Theorems \ref{intro:quotient} and \ref{intro:Cox} now proceed exactly as in \S\ref{sec:lift}.  Namely, by choosing a local section of the universal torsor of $X'$, on the complement of the $L'_i$, we can lift the fibers of our rational map to $G$-orbits for a canonically induced $G$-action on this open subset of the universal torsor.  We claim that this $G$-action canonically extends to the affine space $\Spec\Cox(X')$.  Indeed, the universal torsor has codimension two complement in this affine envelope, so it suffices to extend the dual action to $\Cox(X')$; but this is automatic since $\varphi_A$ sends $L'_i$ to $L_i$ so the dual action is trivial on the sections of the exceptional divisors and hence is completely determined by its action on the localization inverting all these sections.

Next, we claim that, as before, this $G$-action on the affine space $\Spec\Cox(X')$ induces the structure of a torsor away from a codimension two locus.  The argument is essentially verbatim, with the only question being what the substitute for the boundary divisor invariants should be.  The argument requires merely that these form a separating set of invariants on a codimension 2 complement of affine space.   We claim the following invariants suffice: take the linear invariants for the $G$-action on the localization given by inverting the sections of all exceptional divisors, then clear denominators.   It is straightforward to check that these invariants separate orbits on the locus where at most one exceptional divisor section is allowed to vanish, so in particular on a codimension 2 complement.  Ignoring codimension two loci, the quotient of this additive torsor is a $T_{NS}$-torsor over $X$, since the additive action is normalized by the defining $T_{NS}$-action on the universal torsor, and it has trivial Picard group by $\mathbb{A}^1$-homotopy invariance.  We conclude that this additive quotient is isomorphic in codimension one to the universal torsor of $X$, and hence that $\Cox(X) = \Cox(X')^G$.  The argument now that $X$ is a non-reductive GIT quotient of the affine space $\Spec\Cox(X')$ follows exactly as it did in the proof of Corollary \ref{cor:affinequot}.

Finally, we note that we also obtain an expression for $\Cox(X)$ as the intersection of two finitely generated rings, namely, the polynomial ring $\Cox(X')$ and the ring of linear invariants for the $G$-action on the complement of the exceptional divisor loci (see Equations (6) and (8) in \cite{Muk01} for this intersection in the case that all $L_i$ are points).



\begin{thebibliography}{MMMMM}
\setlength{\baselineskip}{11pt}

\bibitem[AD07]{AD07}Asok, A. and B. Doran. ``On unipotent quotients and some $\mathbb{A}^1$-contractible smooth schemes.'' \emph{Int. Math. Res. Papers} (2007), 1--51.

\bibitem[AD08]{AD08}Asok, A. and B. Doran.  ``Vector bundles on contractible smooth schemes." \emph{Duke Math. J.} \textbf{143} no. 3 (2008), 513--530.

\bibitem[AD09]{AD09}Asok, A. and B. Doran. ``$\mathbb{A}^1$-homotopy types, excision, and solvable quotients.'' \emph{Adv. Math.} \textbf{221} no. 4 (2009), 1144--1190.

\bibitem[BB11]{BB11}Batyrev, V. and M. Blume. ``The functor of toric varieties associated with Weyl chambers and Losev-Manin moduli spaces.'' \emph{Tohoku Math. J.} \textbf{63} no. 2 (2011), 581--604.

\bibitem[BS00]{BS00}Berenstein, A. and R. Sjamaar. ``Coadjoint orbits, moment polytopes, and the Hilbert-Mumford criterion.'' \emph{J. Amer. Math. Soc.} \textbf{13} (2000), no. 2, 433--466.

\bibitem[BCHM10]{BCHM10}Birkar, C., Cascini, P., Hacon, C. and J. McKernan. ``Existence of minimal models for varieties of log general type.'' \emph{J. Amer. Math. Soc.} \textbf{23} no. 2 (2010), 405--490.

\bibitem[BM14]{BM14}Block, F. and D. Maclagan. ``Cox rings of wonderful compactifications.'' In preparation.

\bibitem[Bro12]{Bro12}Brown, F.  ``Mixed Tate motives over $\mathbb{Z}$."  to appear in \emph{Annals of Math.} \textbf{175} no. 1 (2012).

\bibitem[Cas09]{Cas09}Castravet, A. ``The Cox ring of $\M_{0,6}$.'' \emph{Trans. Amer. Math. Soc.} \textbf{361} (2009), 3851--3878.

\bibitem[CT13a]{CT13a}Castravet, A.-M. and J. Tevelev. ``Hypertrees, projections, and moduli of stable rational curves.'' \emph{J. Reine Angew. Math.} \textbf{675} (2013), 121--180.

\bibitem[CT13b]{CT13b}Castravet, A.-M. and J. Tevelev. ``$\M_{0,n}$ is not a Mori Dream Space.'' arXiv:1311.7673.

\bibitem[CFM12]{CFM12}Chen, D., Farkas, G., and I. Morrison. ``Effective divisors on moduli spaces.'' To appear in ``A Celebration of Algebraic Geometry,'' Clay Mathematics Proceedings, Volume published on the occasion of Joe Harris' 60th birthday (2012).

\bibitem[CG97]{CG97}Chriss, N. and V. Ginzburg. ``Representation Theory and Complex Geometry.'' Birkh\"{a}user, 1997.

\bibitem[CT76]{CT76}Colliot-Th\'el\`ene, J. and J. Sansuc. ``Torseurs sous des groupes de type multiplicatif; applications \`a l'\'etude des points rationnels de certaines vari\'et\'es alg\'ebriques.'' \emph{C. R. Acad. Sci. Paris S\'er. A-B} \textbf{282} no. 18 (1976), A1113--A1116. 

\bibitem[CT79]{CT79}Colliot-Th\'el\`ene, J. and J. Sansuc. ``La descente sur les vari\'et\'es rationnelles.'' \emph{Journ\'ees de g\'eom\'etrie alg\'ebrique d'Angers} (1979) (A. Beauville, ed.), Sijthoff \& Noordhoff, Alphen aan den Rijn, 1980, 223--237.

\bibitem[CT87]{CT87}Colliot-Th\'el\`ene, J. and J. Sansuc. ``La descente sur les vari\'et\'es rationnelles II.'' \emph{Duke Math. J.} \textbf{54} no. 2 (1987), 375--492. 

\bibitem[Cox95]{Cox95}Cox, D. ``The homogeneous coordinate ring of a toric variety.'' \emph{J. Algebr. Geom.} \textbf{4} (1995), 17--50.

\bibitem[CK99]{CK99}Cox, D. and Katz, S.  \emph{Mirror Symmetry and Algebraic Geometry}.  Mathematical Surveys and Monographs \textbf{68}, AMS, 1999.

\bibitem[DP95]{DP95}De Concini, C. and C. Procesi. ``Wonderful models of subspace arrangements,'' \emph{Selecta Math.} (N.S.) \textbf{1} no. 3 (1995), 459--494.

\bibitem[Del12]{Del12}Deligne, P.  ``Multizetas, d'apres Francis Brown.'' Seminaire Bourbaki, Janvier 2012, 64eme annee, 2011-2012, Exp. 1048.

\bibitem[DK08]{DK08}Derksen, H. and G. Kemper. ``Computing invariants of algebraic groups in arbitrary characteristic.'' \emph{Adv. Math.} \textbf{217} no. 5 (2008), 2089--2129.

\bibitem[DHK14]{DHK14}Doran, B., Hawes, T., and Kirwan, F.  ``Geometric invariant theory for affine algebraic groups.''  In preparation.

\bibitem[DGJ14]{DGJ14}Doran, B., Giansiracusa, N., and D. Jensen.  ``A simplicial approach to effective divisors in $\M_{0,n}$.''  arXiv:1401.0350.

\bibitem[DK07]{DK07}Doran, B. and F. Kirwan. ``Towards non-reductive geometric invariant theory.'' \emph{Pure and applied mathematics quarterly} \textbf{3} (2007), 61--105.

\bibitem[EHKR10]{EHKR10} Etinghof, P., Henriques, J. Kamnitzer, and E. Rains. ``The cohomology ring of the real locus of the moduli space of stable curves of genus 0 with marked points.'' \emph{Annals of Math.} \textbf{171} no. 2 (2010), 731--777. 

\bibitem[FMSS95]{FMSS95}Fulton, W., MacPherson, R., Sotile, F., and B. Sturmfels.  ``Intersection theory on spherical varieties.'' \emph{J. Algebr. Geom.} \textbf{4} (1995), 181--193.

\bibitem[HK00]{HK00} Hu, Y. and S. Keel. ``Mori dream spaces and GIT.'' \emph{Michigan Math. J.} \textbf{48} no. 1 (2000), 331--348.

\bibitem[GKZ94]{GKZ94}Gelfand, I.~M., Kapranov, M., and A. Zelevinsky.  \emph{Discriminants, resultants, and multidimensional determinants}.  Birkh\"{a}user, 1994.

\bibitem[GKM02]{GKM02}Gibney, A., Keel, S. and I. Morrison. ``Towards the ample cone of the moduli space of curves.'' \emph{J. Amer. Math. Soc.} \textbf{15} no. 2 (2002), 273--294.

\bibitem[GM03]{GM03}Goncharov, A. and Yu. Manin. ``Multiple $\zeta$-motives and moduli spaces
$\M_{0,n}$''. \emph{Compos. Math.} \textbf{140} no. 1 (2004), 1--14.

\bibitem[GP93]{GP93}Greuel, G.-M. and G. Pfister.  ``Geometric quotients of unipotent group actions.'' \emph{Proc. 
London Math. Soc.} \textbf{67} no. 3 (1993), 75--105.

\bibitem[GIT94]{GIT}Mumford, D., Fogarty, J. and F.C. Kirwan. \textit{Geometric Invariant Theory.}  Third Edition.  Springer, 1994.

\bibitem[Has03]{Has03}Hassett, B. ``Moduli spaces of weighted pointed stable curves.''  \textit{Adv.\ Math.\ } \textbf{173} no. 2 (2003), 316--352.

\bibitem[Kap93a]{Kap93a}Kapranov, M. ``Veronese curves and Grothendieck-Knudsen moduli space $\M_{0,n}$.'' \emph{J. Algebraic Geom.} \textbf{2} (1993), 239--262.

\bibitem[Kap93b]{Kap93b}Kapranov, M.  ``Chow quotients of Grassmannians, I.''  \textit{Adv. Sov. Math.} \textbf{16} no. 2 (1993), 29--110.

\bibitem[Kho01]{Kho01}Khoury, J. ``On some properties of elementary derivations in dimension six.'' \emph{J. Pure Appl. Algebra} \textbf{156} no. 1(2001), 69--79.

\bibitem[Kho04]{Kho04}Khoury, J. ``A note on elementary derivations.'' \emph{Serdica Math. J.} \textbf{30} no. 4 (2004), 549--570.

\bibitem[Kur04]{Kur04}Kuroda, S. ``A generalization of Roberts' counterexample to the fourteenth problem of Hilbert.'' \emph{Tohoku Math. J.} (2) \textbf{56} no. 4 (2004), 501--522.

\bibitem[Kur09]{Kur09}Kuroda, S. ``A simple proof of Nowicki's conjecture on the kernel of an elementary derivation.'' \emph{Tokyo J. Math.} \textbf{32} no. 1 (2009), 247--251.

\bibitem[LY13]{LY13}Lian, B. and S.-T. Yau.  ``Period integrals of CY and general type complete intersections." \emph{Inventiones Math.} \textbf{191} no. 1 (2013), 35--89.

\bibitem[LM00]{LM00}Losev, A. and Y. Manin. ``New moduli spaces of pointed curves and pencils of flat connections.'' \emph{Michigan Math. J.} \textbf{48} (2000), 443--472.


\bibitem[Muk01]{Muk01}Mukai, S. ``Counterexample to Hilbert's fourteenth problem for the 3-dimensional additive group.'' RIMS Preprint 1343 (2001).

\bibitem[Mum77]{Mum77}Mumford, D. ``Stability of projective varieties.'' \emph{Enseignement Math.} \textbf{23} (1977), 39-110.


\bibitem[Pix11]{Pix11} Pixton, A. ``A non-boundary nef divisor on $\M_{0,12}$.'' arXiv:1112.5512.

\bibitem[PV94]{PV94}Popov, V. and E. Vinberg. \emph{Invariant theory}, in \emph{Algebraic geometry IV, Encyclopaedia of Mathematical Sciences} \textbf{55}, 1994. 

\bibitem[Res10]{Res10}Ressayre, N. ``Geometric invariant theory and generalized eigenvalue problem."  \emph{Inventiones mathematicae} \textbf{180} no. 2 (2010), 389--441.

\bibitem[Tan07]{Tan07}Tanimoto, R. ``A note on Hilbert's fourteenth problem for monomial elementary derivations.'' \emph{Affine algebraic geometry}, Osaka Univ. Press, Osaka, 2007, 437--448.

\bibitem[Tha96]{Tha96}Thaddeus, M. ``Geometric invariant theory and flips.'' \textit{J. Amer. Math. Soc.} \textbf{9} no. 3 (1996), 691-723.

\bibitem[vdE00]{vdE00}van den Essen, A. \emph{Polynomial automorphisms and the Jacobian conjecture.} Progress in Mathematics \textbf{190}, Birkha\"user Verlag, Basel, 2000.

\bibitem[Ver02]{Ver02}Vermeire, P. ``A counterexample to Fulton's Conjecture on $\M_{0,n}$.'' \emph{J. of Algebra} \textbf{248} (2002), 780--784.

\bibitem[CDGW13]{CDGW13}Walter, M., Doran, B., Gross, D., and M. Christandl.  ``Entanglement polytopes: Multi-party entanglement from single particle information.''  \emph{Science}, 2013.

\bibitem[Zag12]{Zag12}Zagier, D. ``Evaluation of the multiple zeta values $\zeta(2,\ldots,2,3,2,\ldots,2)$.'' \emph{Annals of Math.} \textbf{175} (2012), 977--1000.

\end{thebibliography}
\end{document}